\documentclass{amsart}
\usepackage{amssymb}
\usepackage{amsfonts}

\newtheorem{theorem}{Theorem}
\theoremstyle{plain}

\newtheorem{corollary}{Corollary}

\newtheorem{definition}{Definition}
\newtheorem{example}{Example}

\newtheorem{lemma}{Lemma}

\newtheorem{remark}{Remark}

\numberwithin{equation}{section}
\input{tcilatex}

\begin{document}
\title{A new generalization of fuzzy ideals in LA-semigroups}
\author{S. Abdullah}
\address{Department of mathematics, Quaid-i-Azam university, Islamabad,
Pakistan}
\email{saleemabdullah81@yahoo.com}
\author{M Atique Khan}
\address{Department of mathematics, Quaid-i-Azam university, Islamabad,
Pakistan}
\email{atiquekhan\_16@yahoo.com}
\author{M. Aslam}
\address{Department of mathematics, Quaid-i-Azam university, Islamabad,
Pakistan}
\email{draslamqau@yahoo.com}

\begin{abstract}
In this article, the concept of $\ \left( \in _{\gamma },\in _{\gamma }\vee
q_{\delta }\right) $-fuzzy LA-subsemigroups, $\left( \in _{\gamma },\in
_{\gamma }\vee q_{\delta }\right) $-fuzzy left(right) ideals, $\left( \in
_{\gamma },\in _{\gamma }\vee q_{\delta }\right) $-fuzzy generalized
bi-ideals and $\left( \in _{\gamma },\in _{\gamma }\vee q_{\delta }\right) $%
-fuzzy bi-ideals of an LA-semigroup are introduced. The given concept is a
generalization of $\ \left( \in ,\in \vee q\right) $-fuzzy LA-subsemigroups, 
$\left( \in ,\in \vee q\right) $-fuzzy left(right) ideals, $\left( \in ,\in
\vee q\right) $-fuzzy generalized bi-ideals and $\left( \in ,\in \vee
q\right) $-fuzzy bi-ideals of an LA-semigroup. We also give some examples of 
$\left( \in _{\gamma },\in _{\gamma }\vee q_{\delta }\right) $-fuzzy
LA-subsemigroups ( left, right, generalized bi- and bi) ideals of an
LA-semigroup. we prove some fundamental results of these ideals. We
characterize $\left( \in _{\gamma },\in _{\gamma }\vee q_{\delta }\right) $%
-fuzzy left(right) ideals, $\left( \in _{\gamma },\in _{\gamma }\vee
q_{\delta }\right) $-fuzzy generalized bi-ideals and $\left( \in _{\gamma
},\in _{\gamma }\vee q_{\delta }\right) $-fuzzy bi-ideals of an LA-semigroup
by the properties of level sets. The given concept is a generalization of $\
\left( \in ,\in \vee q\right) $-fuzzy LA-subsemigroups, $\left( \in ,\in
\vee q\right) $-fuzzy left(right) ideals, $\left( \in ,\in \vee q\right) $%
-fuzzy generalized bi-ideals and $\left( \in ,\in \vee q\right) $-fuzzy
bi-ideals of an LA-semigroup. We also give some examples of $\left( \in
_{\gamma },\in _{\gamma }\vee q_{\delta }\right) $-fuzzy LA-subsemigroups (
left, right, generalized bi- and bi) ideals of an LA-semigroup.
\end{abstract}

\maketitle

\section{Introduction}

The concept of LA-semigroup was first indroduced by Kazim and Naseerudin 
\cite{qw,MN}. A non-empty set $S$ with binary operation $\ast $ is said to
be an LA-semigroup if identity $(x\ast y)\ast z=(z\ast y)\ast x$ for all $%
x,y,z\in S$ holds. Later, Q. Mushtaq and others have been investigated the
structure further and added many useful results to theory of LA-semigroups
see \cite{QMK,QMY,QMMY}. Ideals of LA-semigroups were defined by Mushtaq and
Khan in his paper \cite{MM}. In \cite{MNA}, Khan and Ahmad characterized
LA-semigroup by their ideals. L. A. Zadeh introduced the fundamental concept
of a fuzzy set \cite{1} in 1965. On the basis of this concept,
mathematicians initiated a natural framework for generalizing some basic
notions of algebra, e.g group theory, set theory, ring theory, topology,
measure theory and semigroup theory etc. The importance of fuzzy technology
in information processing is increasing day by day. In granular computing,
the information is represented in the form of aggregates, called granules.
Fuzzy logic is very useful in modeling the granules as fuzzy sets. Bargeila
and Pedrycz considered this new computing methodology in \cite{22}. Pedrycz
and Gomide in \cite{23} considered the presentation of update trends in
fuzzy set theory and it's applications. Foundations of fuzzy groups are laid
by Rosenfeld in \cite{2}. The theory of fuzzy semigroups was initiated by
Kuroki in his papers \cite{3,4}. Recently, Khan and Khan introduced fuzzy
ideals in LA-semigroups \cite{Fuzz}. In \cite{7} Murali gave the concept of
belongingness of a fuzzy point to a fuzzy subset under a natural equivalence
on a fuzzy subset. In \cite{8} the idea of quasi-coincidence of a fuzzy
point with a fuzzy set is defined. These two ideas played a vital role in
generating some different types of fuzzy subgroups. Using these ideas Bhakat
and Das \cite{9,10} gave the concept of $\left( \alpha ,\beta \right) $%
-fuzzy subgroups, where $\alpha ,\beta \in \{\in ,q,\in \vee q,\in \wedge q\}
$ and $\alpha \neq \in \wedge q$. These fuzzy subgroups are further studied
in \cite{11,12}. The concept of $\left( \in ,\in \vee q\right) $-fuzzy
subgroups is a viable generalization of Rosenfeld's fuzzy subgroups, $\left(
\in ,\in \vee q\right) $-fuzzy subrings and ideals are defined In \cite{13},
S.K. Bhakat and P. Das introduced the $\left( \in ,\in \vee q\right) $-fuzzy
subrings and ideals. Davvaz gave the concept of $\left( \in ,\in \vee
q\right) $-fuzzy subnearrings and ideals of a near ring in \cite{14}. Jun
and Song initiated the study of $\left( \alpha ,\beta \right) $-fuzzy
interior ideals of a semigroup in \cite{15}. In \cite{16} Kazanci and Yamak
studied $\left( \in ,\in \vee q\right) $-fuzzy bi-ideals of a semigroup. In 
\cite{17} regular semigroups are characterized by the properties of $\left(
\in ,\in \vee q\right) $-fuzzy ideals. Aslam et al defined generalized fuzzy 
$\Gamma $-ideals in $\Gamma $-LA-semigroups \cite{222}. In \cite{221},
Abdullah et al give new generalization of fuzzy normal subgroup and fuzzy
coset of groups. Generalizing the idea of the quasi-coincident of a fuzzy
point with a fuzzy subset Jun \cite{19} defined $\left( \in ,\in \vee
q_{k}\right) $-fuzzy subalgebras in BCK/BCI-algebras. In \cite{26}, $\left(
\in ,\in \vee q_{k}\right) $-fuzzy ideals of semigroups are introduced.
Further generalizing the concept, $\left( \in ,\in \vee q_{k}\right) $, J.
Zhan and Y. Yin defined $\left( \in _{\gamma },\in _{\gamma }\vee q_{\delta
}\right) $-fuzzy ideals of near rings \cite{25}. In \cite{24}, $\left( \in
_{\gamma },\in _{\gamma }\vee q_{\delta }\right) $-fuzzy ideals of
BCI-algebras are introduced.

In this article, the concept of $\ \left( \in _{\gamma },\in _{\gamma }\vee
q_{\delta }\right) $-fuzzy LA-subsemigroups, $\left( \in _{\gamma },\in
_{\gamma }\vee q_{\delta }\right) $-fuzzy left(right) ideals, $\left( \in
_{\gamma },\in _{\gamma }\vee q_{\delta }\right) $-fuzzy generalized
bi-ideals and $\left( \in _{\gamma },\in _{\gamma }\vee q_{\delta }\right) $%
-fuzzy bi-ideals of an LA-semigroup are introduced. The given concept is a
generalization of $\ \left( \in ,\in \vee q\right) $-fuzzy LA-subsemigroups, 
$\left( \in ,\in \vee q\right) $-fuzzy left(right) ideals, $\left( \in ,\in
\vee q\right) $-fuzzy generalized bi-ideals and $\left( \in ,\in \vee
q\right) $-fuzzy bi-ideals of an LA-semigroup. We also give some examples of 
$\left( \in _{\gamma },\in _{\gamma }\vee q_{\delta }\right) $-fuzzy
LA-subsemigroups ( left, right, generalized bi- and bi) ideals of an
LA-semigroup. we prove some fundamental results of these ideals. We
characterize $\left( \in _{\gamma },\in _{\gamma }\vee q_{\delta }\right) $%
-fuzzy left(right) ideals, $\left( \in _{\gamma },\in _{\gamma }\vee
q_{\delta }\right) $-fuzzy generalized bi-ideals and $\left( \in _{\gamma
},\in _{\gamma }\vee q_{\delta }\right) $-fuzzy bi-ideals of an LA-semigroup
by the properties of level sets. We prove that for $\mu $ be a fuzzy subset
of $S$. Then, the following hold (i) $\mu $ is an $\left( \in _{\gamma },\in
_{\gamma }\vee q_{\delta }\right) $-fuzzy ideal of $S$ if and only if $\ \mu
_{r}^{\gamma }(\neq \phi )$ is an ideal of $S$ for all $r$ $\in (\gamma
,\delta ].$ (ii) If $2\delta =1+\gamma ,$ then $\mu $ is an $\left( \in
_{\gamma },\in _{\gamma }\vee q_{\delta }\right) $-fuzzy ideal of of $S$ if
and only if $\ \mu _{r}^{\delta }(\neq \phi )$ is an ideal of $S$ for all $r$
$\in (\delta ,1].$ (iii) If $2\delta =1+\gamma ,$ then $\mu $ is an $\left(
\in _{\gamma },\in _{\gamma }\vee q_{\delta }\right) $-fuzzy ideal of of $S$
if and only if $\ [\mu ]_{r}^{\delta }(\neq \phi )$ is an ideal of $S$ for
all $r$ $\in (\gamma ,1].$ (iv) $\mu $ is an $\left( \in _{\gamma },\in
_{\gamma }\vee q_{\delta }\right) $-fuzzy ideal of $S$ if and only if $\
U(\mu ;r)(\neq \phi )$ is an ideal of $S$ for all $r$ $\in (\gamma ,\delta
]. $ Similarly, we prove that for $\left( \in _{\gamma },\in _{\gamma }\vee
q_{\delta }\right) $-fuzzy generalized bi-ideals, $\left( \in _{\gamma },\in
_{\gamma }\vee q_{\delta }\right) $-fuzzy bi-ideals and $\left( \in _{\gamma
},\in _{\gamma }\vee q_{\delta }\right) $-fuzzy quasi-ideals of an
LA-semigroup.

\section{Preliminaries}

An LA-subsemigroup of $S$ \ means a non-empty subset $A$ of $S$ such that $%
A^{2}\subseteq A.$ By a left (right) ideals of $S$ we mean a non-empty
subset $I$ of $S$ such that $SI\subseteq I(IS\subseteq I).$ An ideal $I$ is
said to be two sided or simply ideal if it is both left and right ideal. An
LA-subsemigroup $A$ is called bi-ideal if $\left( BS\right) B\subseteq A.$ A
non-empty subset $B$ is called generalized bi-ideal if $\left( BS\right)
B\subseteq A.$ A non-empty subset $Q$ is called a quasi-ideal if $QS\cap
SQ\subseteq Q.$ A non-empty subset $A$ is called interior ideal if it is
LA-subsemigroup of $S$ and $\left( SA\right) S\subseteq A.$ An LA-semigroup $%
S$ is called regular if for each $a\in S$ there exists $x\in S$ such that $%
a=(ax)a.$ An LA-semigroup $S$ is called intra-regular if for each $a\in S$
there exist $x,y\in S$ such that $a=(xa^{2})y.$ In an LA-semigroup $S,$ the
following law hold, (1) $\left( ab\right) c=\left( ab\right) c,$ for all $%
a,b,c\in S.$ $\left( 2\right) $ $\left( ab\right) \left( cd\right) =\left(
ac\right) \left( bd\right) ,$ for all $a,b,c,d\in S.$ If an LA-semigroup $S$
has a left identity $e,$ then the following law holds, (3) $\left( ab\right)
\left( cd\right) =\left( db\right) \left( ca\right) ,$ for all $a,b,c,d\in
S. $ (4) $a(bc)=b(ac),$ for all $a,b,c\in S.$

A fuzzy subset $\mu $ of the form%
\begin{equation*}
\mu (y)=\left\{ 
\begin{array}{c}
t(\neq 0)\text{ if }y=x \\ 
0\text{ \ \ \ \ \ \ \ if }y\neq x\text{\ }%
\end{array}%
\right.
\end{equation*}%
is said to be a fuzzy point with support $x$ and value $t$ and is denoted by 
$x_{t}$.A fuzzy point $x_{t}$ is said to be "belong
to"(res.,"quasicoincident with") a fuzzy set $\mu ,$ written as $x_{t}\in
\mu ($repectively$,x_{t}q\mu )$ if $\mu (x)\geq t$ $($repectively, $\mu
(x)+t>1).$ We write $x_{t}\in \vee q\mu $ if $x_{t}\in \mu $ or $x_{t}q\mu .$%
If $\mu (x)<t($respectively, $\mu (x)+t\leq 1),$ then we write $x_{t}%
\overline{\in }\mu ($repectively,$x_{t}\overline{q}\mu ).$ We note that $%
\overline{\in \vee q}$ means that $\in \vee q$ does not hold. Generalizing
the concept of $x_{t}q\mu ,$ Y. B. Jun \cite{19} defined $x_{t}q_{k}\mu ,$
where $k\in \lbrack 0,1]$ as $x_{t}q_{k}\mu $ if $\mu (x)+t+\dot{k}>1$ and $%
x_{t}\in \vee q_{k}\mu $ if $x_{t}\in \mu $ or $x_{t}q_{k}\mu .$

A fuzzy subset $\mu $ of $S$ is a function $\mu :S\rightarrow \lbrack 0,1].$
For any two fuzzy subsets $\mu $ and $\nu $ of $S,$ $\mu \subseteq \nu $
means $\mu (x)\leq \nu (x)$ for all $x$ in $S.$ The fuzzy subsets $\mu \cap
\nu $ and $\mu \cup \nu $ of $S$ are defined as 
\begin{eqnarray*}
\left( \mu \cap \nu \right) (x) &=&\min \{\mu (x),\nu (x)\}=\mu (x)\wedge
\nu (x) \\
\left( \mu \cup \nu \right) (x) &=&\max \{\mu (x),\nu (x)\}=\mu (x)\vee \nu
(x)
\end{eqnarray*}%
for all $x$ in $S.$ If $\{\mu _{i}\}_{i\in I}$ is a faimly of fuzzy subsets
\ of $S,$ then $\tbigwedge\limits_{i\in I}\mu _{i}$ and $\bigvee\limits_{i%
\in I}\mu _{i}$ are fuzzy subsets of $S$ defined by%
\begin{eqnarray*}
\left( \tbigwedge\limits_{i\in I}\mu _{i}\right) (x) &=&\min \{\mu
_{i}\}_{i\in I} \\
\left( \bigvee\limits_{i\in I}\mu _{i}\right) (x) &=&\max \{\mu _{i}\}_{i\in
I}
\end{eqnarray*}%
For any two subsets $\mu $ and $\nu $ of $S,$ the product $\mu \circ \nu $
is defined as%
\begin{equation*}
\left( \mu \circ \nu \right) (x)=\QDATOPD\{ . {\bigvee\limits_{x=yz}\{\mu
(y)\wedge \nu (z)\text{, if there exist }y,z\in S,\text{ such that }x=yz}{0%
\text{ \ \ \ \ \ \ \ \ \ \ \ \ \ \ \ \ \ \ \ \ \ \ \ \ \ \ \ \ \ \ \ \ \ \ \
\ \ \ \ \ \ \ \ \ \ \ \ \ \ \ \ \ \ \ \ \ \ \ \ otherwise \ \ \ \ \ \ \ \ \
\ }}
\end{equation*}

\begin{definition}
\cite{Fuzz} A fuzzy subset $\mu $ of an LA-semigroup $S$ is called fuzzy
LA-subsemigroup $S$ if $\mu (xy)\geq \mu (x)\wedge \mu (y)$ for all $x,y\in
S.$
\end{definition}

\begin{definition}
\cite{Fuzz} A fuzzy subset $\mu $ of an LA-semigroup $S$ is called fuzzy
left(right) ideal of $S$ if $\mu (xy)\geq \mu (y)(\mu (xy)\geq \mu (x))$ for
all $x,y\in S.$
\end{definition}

\begin{definition}
\cite{Fuzz} An LA-subsemigroup $\mu $ of an LA-semigroup $S$ is called fuzzy
bi-ideal of $S$ if $\mu (\left( xy\right) z)\geq \mu (x)\wedge \mu (z)$ for
all $x,y\in S.$
\end{definition}

\begin{definition}
\cite{Fuzz} A fuzzy subset $\mu $ of an LA-semigroup $S$ is called fuzzy
generalized bi-ideal of $S$ if $\mu (\left( xy\right) z)\geq \mu (x)\wedge
\mu (z)$ for all $x,y\in S.$
\end{definition}

\begin{definition}
\cite{Fuzz}Let $\mu $ be a fuzzy subset of an LA-semigroup $S,$ then for all 
$t\in (0,1],$ the set $\mu _{t}=\{x\in S\left\vert \mu (x)\geq t\right. \}$
is called a level subset of $S.$
\end{definition}

\section{$\left( \in _{\protect\gamma },\in _{\protect\gamma }\vee q_{%
\protect\delta }\right) $-FUZZY IDEALS}

Generalizing the notion of $\left( \in ,\in \vee q\right) ,$ in \cite{24,25} 
$\left( \in _{\gamma },\in _{\gamma }\vee q_{\delta }\right) $-fuzzy ideals
of near rings and BCI-algebras are introduced. Let $\gamma ,\delta \in
\lbrack 0,1]$ be such that $\gamma <\delta .$ For fuzzy point $x_{t}$ and
fuzzy subset $\mu $ of $S,$ we say

$(i)$\ $x_{t}\in _{\gamma }\mu $ if $\mu (x)\geq t>\gamma .$

$(ii)\ x_{t}q_{\delta }\mu $ \ if $\mu (x)+t>2\delta .$

$(iii)\ x_{t}\in _{\gamma }\vee q_{\delta }\mu $ if $x_{t}\in _{\gamma }\mu $
or $x_{t}q_{\delta }\mu .$

$(iv)\ x_{t}\overline{\in }_{\gamma }\vee \overline{q}_{\delta }\mu $ if $%
x_{t}\overline{\in }_{\gamma }\mu $ or $x_{t}\overline{q}_{\delta }\mu $.

In this section, we introduce the concept of $\left( \in _{\gamma },\in
_{\gamma }\vee q_{\delta }\right) $-fuzzy LA-subsemigroup, $\left( \in
_{\gamma },\in _{\gamma }\vee q_{\delta }\right) $-fuzzy left(right) ideal, $%
\left( \in _{\gamma },\in _{\gamma }\vee q_{\delta }\right) $-fuzzy
generalized bi-ideal and $\left( \in _{\gamma },\in _{\gamma }\vee q_{\delta
}\right) $-fuzzy bi-ideal of an LA-semigroup $S$. We also study some basic
properties of these ideals.

\begin{definition}
A fuzzy subset $\mu $ of an LA-semigroup $S$ is called an $\left( \in
_{\gamma },\in _{\gamma }\vee q_{\delta }\right) $-fuzzy LA-subsemigroup of $%
S$ if for all $a,b\in S$ and $t,r\in (\gamma ,1]$, $\ x_{t},y_{r}\in
_{\gamma }\mu $ implies that $\left( ab\right) _{t\wedge r}\in _{\gamma
}\vee q_{\delta }\mu .$
\end{definition}

\begin{remark}
Every fuzzy LA-subsemigroup and every $\left( \in ,\in \vee q\right) $-fuzzy
LA-subsemigroup is an $\left( \in _{\gamma },\in _{\gamma }\vee q_{\delta
}\right) $-fuzzy LA-subsemigroup but the converse is not true.
\end{remark}

\begin{example}
Let $S=\{1,2,3,4\}$ be an LA -semigroup with the multiplication defined by
the following Caley table

\begin{equation*}
\begin{tabular}{l|llll}
${\ast }$ & $1$ & $2$ & $3$ & $4$ \\ \hline
$1$ & $4$ & $3$ & $3$ & $3$ \\ 
$2$ & $3$ & $3$ & $4$ & $4$ \\ 
$3$ & $3$ & $3$ & $3$ & $3$ \\ 
$4$ & $3$ & $3$ & $3$ & $3$%
\end{tabular}%
\end{equation*}%
Define fuzzy subset $\mu $ of $S$ by

$\mu (1)=0.6,$ $\mu (2)=0.9,$ $\mu (3)=0.7,$ $\mu (4)=0.3.$

Then,

$(i)$ \ \ \ \ $\mu $ is an $\left( \in _{0.2},\in _{0.2}\vee q_{0.3}\right) $%
-fuzzy LA-subsemigroup.

$\left( ii\right) $ \ \ \ $\mu $ is not an $\left( \in ,\in \vee q\right) $%
-fuzzy LA-subsemigroup \ because $2_{0.6},3_{0.6}\in \mu ,$ but $\left(
2\ast 3\right) _{0.6}\overline{\in \vee q}\mu .$

$\left( iii\right) $ \ \ $\mu $ is not fuzzy LA-subsemigroup.
\end{example}

The next theorems provide the relationship between $\left( \in _{\gamma
},\in _{\gamma }\vee q_{\delta }\right) $-fuzzy LA-subsemigroups of $S$ and
crisp LA-subsemigroups.

\begin{theorem}
Let $A$ be a non-empty subset of $S.$ Then, $A$ is an LA-subsemigroup of $S$
if and only if the fuzzy subset $\mu $\ of $S$ defined by%
\begin{equation*}
\mu (a)=\QDATOPD\{ . {\geq \delta \text{ if }a\in A}{\leq \gamma \text{ if }%
a\notin A}
\end{equation*}%
\ is an $\left( \in _{\gamma },\in _{\gamma }\vee q_{\delta }\right) $-fuzzy
LA-subsemigroup of $S.$
\end{theorem}

\begin{proof}
Let $A$ be an LA-subsemigroup of $S$ and $a,b\in S$ and $t,r$ $\in (\gamma
,1]$, such that $a_{t},b_{r}\in _{\gamma }\mu .$ Then, $\mu (a)\geq t>\gamma 
$ and $\mu (b)\geq r>\gamma .$ Thus, $\mu (a)\geq \delta $ and $\mu (b)\geq
\delta $. Hence, $a,b\in A.$ Since $A$ is an LA-subsemigroup of $S,$ so we
have $ab\in A,$ which implies that $\mu (ab)\geq \delta .$ If $t\wedge r\leq
\delta ,$ then $\mu (ab)\geq \delta \geq t\wedge r>\gamma $ and so $\left(
ab\right) _{t\wedge r}\in _{\gamma }\vee q_{\delta }\mu .$ If $t\wedge
r>\delta $ then, $\mu (ab)+t\wedge r>\mu (ab)+\delta >\delta +\delta ,$
which implies that $\mu (ab)+t\wedge r>2\delta $ and so $\left( ab\right)
_{t\wedge r}\in _{\gamma }\vee q_{\delta }\mu .$ Hence, $\mu $ is an $\left(
\in _{\gamma },\in _{\gamma }\vee q_{\delta }\right) $-fuzzy LA-subsemigroup
of $S.$

Conversely, assume that $\mu $ is an $\left( \in _{\gamma },\in _{\gamma
}\vee q_{\delta }\right) $-fuzzy LA-subsemigroup of $S.$ Let $a,b\in A,$
then $\mu (a)\geq \delta $ and $\mu (b)\geq \delta .$ By hypothesis, $\mu
(ab)\vee \gamma \geq \mu (a)\wedge \mu (b)\wedge \delta ,$ which implies
that $\mu (ab)\vee \gamma \geq \delta \wedge \delta \wedge \delta =\delta ,$
that is $\mu (ab)\vee \gamma \geq \delta .$ Since $\gamma <\delta ,$
therefore $\mu (ab)\geq \delta .$ Hence, $ab\in A.$ Thus, $A$ is an
LA-subsemigroup of $S.$
\end{proof}

\begin{corollary}
Let $A$ be a non-empty subset of an LA-semigroup $S.$ Then, $A$ is an
LA-subsemigroup of $S$ if and only if $\chi _{A},$ the characteristic
function of $A$ is an $\left( \in _{\gamma },\in _{\gamma }\vee q_{\delta
}\right) $-fuzzy LA-subsemigroup of $S.$
\end{corollary}

\begin{theorem}
Let $2\delta =1+\gamma $ and $A$ be an LA-subsemigroup of $S.$ Then, the
fuzzy subset $\mu $\ of $S$ defined by%
\begin{equation*}
\mu (a)=\QDATOPD\{ . {\geq \delta \text{ if }a\in A}{\leq \gamma \text{ if }%
a\notin A}
\end{equation*}%
\ is an $\left( q_{\delta },\in _{\gamma }\vee q_{\delta }\right) $-fuzzy
LA-subsemigroup of $S.$
\end{theorem}

\begin{proof}
Let $2\delta =1+\gamma $ and $A$ be an LA-subsemigroup of $S.$ Let $a,b\in S$
and $t,r$ $\in (\gamma ,1]$, such that $a_{t},b_{r}q_{\delta }\mu $, then $%
\mu (a)+t>2\delta $ and $\mu (b)+r>2\delta .$ This implies that $\mu
(a)>2\delta -t\geq 2\delta -1=\gamma $ and $\mu (b)>2\delta -r\geq 2\delta
-1=\gamma ,$ thus $ab\in A,$ which implies that $\mu (ab)\geq \delta .$ If $%
t\wedge r\leq \delta ,$ then $\mu (ab)\geq \delta \geq t\wedge r>\gamma $
and so $\left( ab\right) _{t\wedge r}\in _{\gamma }\vee q_{\delta }\mu .$ If 
$t\wedge r>\delta ,$ then $\mu (ab)+t\wedge r>\mu (ab)+\delta >\delta
+\delta ,$ which implies that $\mu (ab)+t\wedge r>2\delta $ and so $\left(
ab\right) _{t\wedge r}\in _{\gamma }\vee q_{\delta }\mu .$ Hence, $\mu $ is
an $\left( q_{\delta },\in _{\gamma }\vee q_{\delta }\right) $-fuzzy
LA-subsemigroup of $S.$
\end{proof}

\begin{theorem}
A fuzzy subset $\mu $ of an LA-semigroup $S$ is an $\left( \in _{\gamma
},\in _{\gamma }\vee q_{\delta }\right) $-fuzzy LA-subsemigroup of $S$ if
and only if $\mu (ab)\vee \gamma \geq \mu (a)\wedge \mu (b)\wedge \delta $
for all $a,b\in S.$
\end{theorem}

\begin{proof}
Let $\mu $ is an $\left( \in _{\gamma },\in _{\gamma }\vee q_{\delta
}\right) $-fuzzy LA-subsemigroup of $S.$ To show that $\mu (ab)\vee \gamma
\geq \mu (a)\wedge \mu (b)\wedge \delta ,$ for all $a,b\in S,$ we discuss
the following two cases:

$\left( a\right) $ \ $\mu (a)\wedge \mu (b)\leq \delta .$

$\left( b\right) $ $\mu (a)\wedge \mu (b)>\delta .$

Case(a): If there exist $a,b\in S$ such that $\mu (ab)\vee \gamma <\mu
(a)\wedge \mu (b)\wedge \delta .$ Then, $\mu (ab)\vee \gamma <t<\mu
(a)\wedge \mu (b),$ which implies that $a_{t}\in _{\gamma }\mu ,$ $b_{t}\in
_{\gamma }\mu $ but $\left( ab\right) _{t}\overline{\in _{\gamma }\vee
q_{\delta }}\mu ,$ which is a \ contradiction. Hence, $\mu (ab)\vee \gamma
\geq \mu (a)\wedge \mu (b)\wedge \delta .$ for all $a,b\in S.$ Case(b): If
there exist $a,b\in S$ such that $\mu (ab)\vee \gamma <\mu (a)\wedge \mu
(b)\wedge \delta .$ Then, $\mu (ab)\vee \gamma <t<\delta $ and so $a_{\delta
}\in _{\gamma }\mu ,$ $b_{\delta }\in _{\gamma }\mu $ but $\left( ab\right)
_{\delta }\overline{\in _{\gamma }\vee q_{\delta }}\mu ,$ which is a\
contradiction. Hence, $\mu (ab)\vee \gamma \geq \mu (a)\wedge \mu (b)\wedge
\delta $ for all $a,b\in S.$

Conversely, assume that $\mu (ab)\vee \gamma \geq \mu (a)\wedge \mu
(b)\wedge \delta $ for all $a,b\in S.$ We are to show that $\mu $ is an $%
\left( \in _{\gamma },\in _{\gamma }\vee q_{\delta }\right) $-fuzzy
LA-subsemigroup of $S.$ For this let $a_{t},b_{r}\in _{\gamma }\mu .$ Then, $%
\mu (a)\geq t>\gamma $ and $\mu (b)\geq r>\gamma .$ Now by hypothesis, we
have $\mu (ab)\vee \gamma \geq \mu (a)\wedge \mu (b)\wedge \delta \geq
t\wedge r\wedge \delta .$ If $t\wedge r\geq \delta ,$ then $\mu (ab)\vee
\gamma \geq \delta >\gamma $ and so $\mu (ab)>\gamma ,$ that is $\mu
(ab)+t\wedge r>2\delta .$ Hence, $\left( ab\right) _{t\wedge r}\in _{\gamma
}\vee q_{\delta }\mu .$ If $t\wedge r<\delta ,$ then $\mu (ab)\vee \gamma
\geq $ $t\wedge r$ and so $\mu (ab)\geq $ $t\wedge r>\gamma .$ Hence, $%
\left( ab\right) _{t\wedge r}\in _{\gamma }\vee q_{\delta }\mu .$
\end{proof}

\begin{remark}
For any $\left( \in _{\gamma },\in _{\gamma }\vee q_{\delta }\right) $-fuzzy
LA-subsemigroup $\mu $ of $S,$ we can conclude

(i) If $\gamma =0$ and $\delta =1,$ then $\mu $ is a fuzzy LA-subsemigroup
of $S.$

(ii) If $\gamma =0$ and $\delta =0.5,$ then $\mu $ is an $\left( \in ,\in
\vee q\right) $- fuzzy LA-subsemigroup of $S.$
\end{remark}

\begin{theorem}
Let $\mu $ be a fuzzy subset of $S.$ Then,

(i) $\mu $ is an $\left( \in _{\gamma },\in _{\gamma }\vee q_{\delta
}\right) $-fuzzy LA-subsemigroup of $S$ if and only if $\ \mu _{r}^{\gamma
}(\neq \phi )$ is an LA-subsemigroup of $S$ for all $r$ $\in (\gamma ,\delta
].$

(ii) If $2\delta =1+\gamma ,$ then $\mu $ is an $\left( \in _{\gamma },\in
_{\gamma }\vee q_{\delta }\right) $-fuzzy LA-subsemigroup of $S$ if and only
if $\mu _{r}^{\delta }(\neq \phi )$ is an LA-subsemigroup of $S$ for all $r$ 
$\in (\delta ,1].$

(iii) If $2\delta =1+\gamma ,$ then $\mu $ is an $\left( \in _{\gamma },\in
_{\gamma }\vee q_{\delta }\right) $-fuzzy LA-subsemigroup of $S$ if and only
if $[\mu ]_{r}^{\delta }(\neq \phi )$ is an LA-subsemigroup of $S$ for all $%
r $ $\in (\gamma ,1].$

(iv) $\mu $ is an $\left( \in _{\gamma },\in _{\gamma }\vee q_{\delta
}\right) $-fuzzy LA-subsemigroup of $S$ if and only if $\ U(\mu ;r)(\neq
\phi )$ is an LA-subsemigroup of $S$ for all $r$ $\in (\gamma ,\delta ].$
\end{theorem}

\begin{proof}
(i) Let $\mu $ be an $\left( \in _{\gamma },\in _{\gamma }\vee q_{\delta
}\right) $-fuzzy LA-subsemigroupof $S$ and $a,b\in \mu _{r}^{\gamma }$, for
all $r\in (\gamma ,\delta ].$ Then, $a_{r},b_{r}\in _{\gamma }\mu $, that is 
$\mu (a)\geq r>\gamma $ and $\mu (b)\geq r>\gamma .$ By hypothesis, $\mu
(ab)\vee \gamma \geq \mu (a)\wedge \mu (b)\wedge \delta ,$ which implies
that $\mu (ab)\vee \gamma \geq r\wedge r\wedge \delta =r\wedge \delta .$
Since $r$ $\in (\gamma ,\delta ],$ so $r\leq \delta .$ Thus, $\mu (ab)\geq
r>\gamma ,$ which implies that $ab\in \ \mu _{r}^{\gamma }.$ Hence,$\ \mu
_{r}^{\gamma }$ is LA-subsemigroup of $S.$

Conversely, assume that $\mu _{r}^{\gamma }(\neq \phi )$ is an
LA-subsemigroup of $S$ for all $r$ $\in (\gamma ,\delta ].$ Let $a,b\in S,$
such that $\mu (ab)\vee \gamma <\mu (a)\wedge \mu (b)\wedge \delta .$ Select 
$r$ $\in (\gamma ,\delta ],$ such that $\mu (ab)\vee \gamma <r\leq \mu
(a)\wedge \mu (b)\wedge \delta .$ Then, $a_{r},b_{r}\in _{\gamma }\mu $ but $%
\left( ab\right) _{r}\overline{\in _{\gamma }\vee q_{\delta }}\mu .$ Which
is a contradiction. Thus, $\mu (ab)\vee \gamma \geq \mu (a)\wedge \mu
(b)\wedge \delta .$ Hence, $\mu $ is an $\left( \in _{\gamma },\in _{\gamma
}\vee q_{\delta }\right) $-fuzzy LA-subsemigroup of $S$.

(ii) Let $\mu $ be an $\left( \in _{\gamma },\in _{\gamma }\vee q_{\delta
}\right) $-fuzzy LA-subsemigroup of $S$ and $a,b\in \mu _{r}^{\delta }$ for
all $r\in (\delta ,1].$ Then, $a_{r},b_{r}q_{\delta }\mu $, that is $\mu
(a)+r>2\delta $ and $\mu (b)+r>2\delta .$ Now, $\mu (a)+r>$ $2\delta $
implies that $\mu (a)>2\delta -r\geq 2\delta -1=\gamma .$ Similarly, $\mu
(b)>\gamma .$ By hypothesis, $\mu (ab)\vee \gamma \geq \mu (a)\wedge \mu
(b)\wedge \delta ,$ so we have $\mu (ab)\geq \mu (a)\wedge \mu (b)\wedge
\delta >\left( 2\delta -r\right) \wedge \left( 2\delta -r\right) \wedge
\delta .$ Since $r$ $\in (\delta ,1].$ So, $r>\delta $ which implies that $%
2\delta -r\leq \delta .$ Thus, $\mu (ab)>2\delta -r$ or $\mu (ab)+r>2\delta
. $ Thus, $ab\in \ \mu _{r}^{\delta }.$ Hence, $\ \mu _{r}^{\delta }$ is
LA-subsemigroup of $S.$

Conversely, assume that $\mu _{r}^{\delta }(\neq \phi )$ is an
LA-subsemigroup of $S$ for all $r$ $\in (\delta ,1].$ Let $a,b\in S$ such
that $\mu (ab)\vee \gamma <\mu (a)\wedge \mu (b)\wedge \delta .$ Select $r$ $%
\in (\gamma ,\delta ]$ such that $\mu (ab)\vee \gamma <r\leq \mu (a)\wedge
\mu (b)\wedge \delta .$ Then, $a_{r},b_{r}\in _{\gamma }\mu $ but $\left(
ab\right) _{r}\overline{\in _{\gamma }\vee q_{\delta }}\mu .$ Which is a
contradiction. Thus, $\mu (ab)\vee \gamma \geq \mu (a)\wedge \mu (b)\wedge
\delta .$ Hence, $\mu $ is an $\left( \in _{\gamma },\in _{\gamma }\vee
q_{\delta }\right) $-fuzzy LA-subsemigroup of $S.$

(iii) Let $\mu $ be an $\left( \in _{\gamma },\in _{\gamma }\vee q_{\delta
}\right) $-fuzzy LA-subsemigroup of $S$ and $a,b\in \ [\mu ]_{r}^{\delta }$,
for all $r\in (\gamma ,1].$ Then, $a_{r},b_{r}\in _{\gamma }\vee q_{\delta
}\mu $, that is $\mu (a)\geq r>\gamma $ or $\mu (a)+r>2\delta $ $\ $ which
implies that $\mu (a)\geq r>\gamma $ or $\mu (a)>2\delta -r\geq 2\delta
-1=\gamma .$ Similarly, $\mu (b)\geq r>\gamma $ or $\mu (b)>2\delta -r\geq
2\delta -1=\gamma .$ By hypothesis, $\mu (ab)\vee \gamma \geq \mu (a)\wedge
\mu (b)\wedge \delta $ which implies that $\mu (ab)\geq \mu (a)\wedge \mu
(b)\wedge \delta .$ Case(i): If $r$ $\in (\gamma ,\delta ],$ then $r\leq
\delta .$ This implies that $2\delta -r\geq \delta \geq r.$ Thus, $\mu
(ab)>r\wedge r\wedge \delta =r>\gamma $ or $\mu (ab)>r\wedge \left( 2\delta
-r\right) \wedge \delta =r>\gamma $ or $\mu (ab)>\left( 2\delta -r\right)
\wedge \left( 2\delta -r\right) \wedge \delta =\delta \geq r>\gamma .$
Hence, $ab\in \ \mu _{r}^{\delta }.$ Case(ii): If $r$ $\in (\delta ,1],$
then $r>\delta $ which implies that $2\delta -r<\delta <r,$ and so $\mu
(ab)>r\wedge r\wedge \delta =\delta $ $>2\delta -r$ or $\mu (ab)>r\wedge
2\delta -r\wedge \delta =2\delta -r$ or $\mu (ab)>\left( 2\delta -r\right)
\wedge \left( 2\delta -r\right) \wedge \delta =2\delta -r.$ Thus, $ab\in \
\mu _{r}^{\delta }.$ Hence,$\ \mu _{r}^{\delta }$ is LA-subsemigroup of $S.$
Conversely, assume that $\mu _{r}^{\delta }(\neq \phi )$ is LA-subsemigroup
of $S$ for all $r$ $\in (\delta ,1].$ Let $a,b\in S,$ such that $\mu
(ab)\vee \gamma <\mu (a)\wedge \mu (b)\wedge \delta .$ Select $r$ $\in
(\gamma ,\delta ]$ such that $\mu (ab)\vee \gamma <r\leq \mu (a)\wedge \mu
(b)\wedge \delta .$ Then, $a_{r}\in _{\gamma }\mu ,$ and $b_{r}\in _{\gamma
}\mu $ but $\left( ab\right) _{r}\overline{\in _{\gamma }\vee q_{\delta }}%
\mu ,$ which is a contradiction. Thus, $\mu (ab)\vee \gamma \geq \mu
(a)\wedge \mu (b)\wedge \delta .$ Hence, $\mu $ is an $\left( \in _{\gamma
},\in _{\gamma }\vee q_{\delta }\right) $-fuzzy LA-subsemigroup of $S$.

(iv) Let $\mu $ be an $\left( \in _{\gamma },\in _{\gamma }\vee q_{\delta
}\right) $-fuzzy LA-subsemigroup of $S$ and $a,b\in \ U(\mu ;r)$ for some $%
r\in (\gamma ,\delta ]$. Then $\mu (a)\geq r$ and $\mu (b)\geq r.$ Since $%
\mu $ is an $\left( \in _{\gamma },\in _{\gamma }\vee q_{\delta }\right) $%
-fuzzy LA-subsemigroup of $S,$ we have $\mu (ab)\vee \gamma \geq \mu
(a)\wedge \mu (b)\wedge \delta \geq r\wedge \delta =r,$ which implies that $%
\mu (ab)\geq r$ (because $r>\gamma $). Thus $ab\in U(\mu ;r).$ Hence $U(\mu
;r)$ is an LA-subsemigroup of $S.$

Conversely, assume that $U(\mu ;r)$ is an LA-subsemigroup of $S$ for all $r$ 
$\in (\gamma ,\delta ].$ Suppose that ther exist $a,b\in S$ such that $\mu
(ab)\vee \gamma <\mu (a)\wedge \mu (b)\wedge \delta .$ Select $r$ $\in
(\gamma ,\delta ]$ such that $\mu (ab)\vee \gamma <r\leq \mu (a)\wedge \mu
(b)\wedge \delta .$ Thus $a\in U(\mu ;r)$ and $b\in U(\mu ;r)$ but $ab\notin
U(\mu ;r),$ which is a contradiction. Hence $\mu (ab)\vee \gamma \geq \mu
(a)\wedge \mu (b)\wedge \delta $ and so $\mu $ is an $\left( \in _{\gamma
},\in _{\gamma }\vee q_{\delta }\right) $-fuzzy LA-subsemigroup of $S$.
\end{proof}

If we take $\gamma =0$ and $\delta =0.5$ in above \ theorem, we can conclude
the following results.

\begin{corollary}
Let $\mu $\ be a fuzzy set of $S.$ Then,

(i) $\mu $ is an $\left( \in ,\in \vee q\right) $-fuzzy LA-subsemigroup of $%
S $ if and only if $\ \mu _{r}(\neq \phi )$ LA-subsemigroup of $S$ for all $%
r $ $\in (0,0.5].$

(ii) $\mu $ is an $\left( \in ,\in \vee q\right) $-fuzzy LA-subsemigroup of $%
S$ if and only if $\ Q(\mu ;r)(\neq \Phi )$ LA-subsemigroup of $S$ for all $%
r $ $\in (0.5,1].$

(iii) $\mu $ is an $\left( \in ,\in \vee q\right) $-fuzzy LA-subsemigroup of 
$S$ if and only if $\ [\mu ]_{r}(\neq \Phi )$ LA-subsemigroup of $S$ for all 
$r$ $\in (0,1].$
\end{corollary}

\begin{lemma}
The intersection of any faimly of $\left( \in _{\gamma },\in _{\gamma }\vee
q_{\delta }\right) $-fuzzy LA-subsemigroups is an $\left( \in _{\gamma },\in
_{\gamma }\vee q_{\delta }\right) $-fuzzy LA-subsemigroup .
\end{lemma}

\begin{proof}
Let $\{\mu _{i}\}_{i\in I}$ \ be a faimly of $\left( \in _{\gamma },\in
_{\gamma }\vee q_{\delta }\right) $-fuzzy LA-subsemigroups of $S$ and $%
x,y\in S.$ Then, 
\begin{eqnarray*}
((\tbigwedge\limits_{i\in I}\mu _{i})(xy))\vee \gamma
&=&(\tbigwedge\limits_{i\in I}\mu _{i}(xy))\vee \gamma \\
&=&(\tbigwedge\limits_{i\in I}((\mu _{i}(xy))\vee \gamma )) \\
&\geq &(\tbigwedge\limits_{i\in I}\{\mu _{i}(x)\wedge \mu _{i}(y)\wedge
\delta \}) \\
&=&(\tbigwedge\limits_{i\in I}\mu _{i}(x))\wedge (\tbigwedge\limits_{i\in
I}\mu _{i}(y))\wedge \delta \\
&=&(\tbigwedge\limits_{i\in I}\mu _{i})(x))\wedge (\tbigwedge\limits_{i\in
I}\mu _{i})(y))\wedge \delta .
\end{eqnarray*}%
Hence, $\tbigwedge\limits_{i\in I}\mu _{i}$ is an $\left( \in _{\gamma },\in
_{\gamma }\vee q_{\delta }\right) $-fuzzy LA-subsemigroups of $S.$
\end{proof}

\begin{definition}
A fuzzy subset $\mu $ of an LA-semigroup $S$ is called \ an \ $\left( \in
_{\gamma },\in _{\gamma }\vee q_{\delta }\right) $- fuzzy left(right) ideal
of $S$ if for all $a,s\in S$ and $t\in (\gamma ,1]$,$\ a_{t}\in _{\gamma
}\mu $ implies that $\left( sa\right) _{t}\in _{\gamma }\vee q_{\delta }\mu $
$\left( \left( as\right) _{t}\in _{\gamma }\vee q_{\delta }\mu \right) .$
\end{definition}

\begin{remark}
Every fuzzy left(right) ideal and every $\left( \in ,\in \vee q\right) $%
-fuzzy left(right) ideal is an $\left( \in _{\gamma },\in _{\gamma }\vee
q_{\delta }\right) $-fuzzy left(right) ideal but the converse is not true.
\end{remark}

\begin{example}
Let $S=\{1,2,3,4\}$ be an LA-semigroup with the following multiplication
table.%
\begin{equation*}
\begin{tabular}{l|llll}
${\ast }$ & $1$ & $2$ & $3$ & $4$ \\ \hline
$1$ & $2$ & $4$ & $3$ & $2$ \\ 
$2$ & $4$ & $4$ & $4$ & $4$ \\ 
$3$ & $4$ & $4$ & $4$ & $4$ \\ 
$4$ & $4$ & $4$ & $4$ & $4$%
\end{tabular}%
\end{equation*}

Define fuzzy subset $\mu $ of $S$ by

$\mu (1)=0.6=$ $\mu (2),$ $\mu (3)=0.4=\mu (4).$

Then,

(i) $\mu $ is an $\left( \in _{0.2},\in _{0.2}\vee q_{0.4}\right) $-fuzzy
left ideal of $S.$

(ii) $\mu $ is not an $\left( \in ,\in \vee q\right) $ fuzzy ideal of $S.$
Because $1_{0.6}\in \mu $ and $2_{0.6}\in \mu $ but $\left( 1\ast 2\right)
_{0.6}\overline{\in \vee q}\mu $

(iii) $\mu $ is not fuzzy ideal of $S.$
\end{example}

\begin{remark}
For any \ $\left( \in _{\gamma },\in _{\gamma }\vee q_{\delta }\right) $%
-fuzzy left(right) ideal $\mu $ of $S$ we can conclude

(i) if $\gamma =0$ and $\delta =1$, then $\mu $ is fuzzy left(right) ideal
of $S.$

(ii) if $\gamma =0$ and $\delta =0.5$ then $\mu $ is an $\left( \in ,\in
\vee q\right) $- fuzzy left(right) ideal of $S.$
\end{remark}

\begin{theorem}
Let $A$ be a non-empty subset of $S.$ Then, $A$ is a left(right) ideal of $S$
if and only if the fuzzy subset $\mu $\ of $S$ defined by%
\begin{equation*}
\mu (a)=\QDATOPD\{ . {\geq \delta \text{ if }a\in A}{\leq \gamma \text{ if }%
a\notin A}
\end{equation*}%
\ is an $\left( \in _{\gamma },\in _{\gamma }\vee q_{\delta }\right) $-fuzzy
left(right) ideal of $S.$
\end{theorem}

\begin{proof}
Proof is similar to the proof of Theorem 4.
\end{proof}

\begin{corollary}
Let $A$ be a non-empty subset of an LA-semigroup $S.$ Then, $A$ is a
left(right) ideal of $S$ if and only if $\chi _{A},$ the characteristic
function of $A$ is an $\left( \in _{\gamma },\in _{\gamma }\vee q_{\delta
}\right) $-fuzzy left(right) ideal of $S.$
\end{corollary}

\begin{theorem}
Let $2\delta =1+\gamma $ and $A$ be a left(right) ideal of $S.$ Then, the
fuzzy subset $\mu $\ of $S$ defined by%
\begin{equation*}
\mu (a)=\QDATOPD\{ . {\geq \delta \text{ if }a\in A}{\leq \gamma \text{ if }%
a\notin A}
\end{equation*}%
\ is an $\left( q_{\delta },\in _{\gamma }\vee q_{\delta }\right) $-fuzzy
left(right) ideal of $S.$
\end{theorem}

\begin{proof}
Proof is similar to the proof of Theorem 5.
\end{proof}

\begin{theorem}
A fuzzy subset $\mu $ of an LA-semigroup $S$ is an $\left( \in _{\gamma
},\in _{\gamma }\vee q_{\delta }\right) $- fuzzy left(right) ideal of an
LA-semigroup $S$ if and only if $\mu \left( sa\right) \vee \gamma \geq \mu
\left( a\right) \wedge \delta $ $\ \left( \mu \left( as\right) \vee \gamma
\geq \mu \left( a\right) \wedge \delta \right) $ for all $a,s\in S.$
\end{theorem}

\begin{proof}
Proof is similar to the proof of Theorem 6.
\end{proof}

\begin{theorem}
Let $\mu $ be a fuzzy subset of $S$ . Then,

(i) $\mu $ is an $\left( \in _{\gamma },\in _{\gamma }\vee q_{\delta
}\right) $-fuzzy left(right) ideal of $S$ if and only if $\ \mu _{r}^{\gamma
}(\neq \phi )$ is left(right) ideal of $S$ for all $r$ $\in (\gamma ,\delta
].$

(ii) If $2\delta =1+\gamma .$ Then, $\mu $ is an $\left( \in _{\gamma },\in
_{\gamma }\vee q_{\delta }\right) $-fuzzy left(right) ideal of $S$ if and
only if $\ \mu _{r}^{\delta }(\neq \phi )$ is left(right) ideal of $S$ for
all $r$ $\in (\delta ,1].$

(iii) If $2\delta =1+\gamma .$ Then, $\mu $ is an $\left( \in _{\gamma },\in
_{\gamma }\vee q_{\delta }\right) $-fuzzy left(right) ideal of $S$ if and
only if $\ [\mu ]_{r}^{\delta }(\neq \phi )$ is left(right) ideal of $S$ for
all $r$ $\in (\gamma ,1].$

(iv) $\mu $ is an $\left( \in _{\gamma },\in _{\gamma }\vee q_{\delta
}\right) $-fuzzy left(right) ideal of $S$ if and only if $\ U(\mu ;r)(\neq
\phi )$ is a left(right) ideal of $S$ for all $r$ $\in (\gamma ,\delta ].$
\end{theorem}

\begin{proof}
Proof is similar to the proof of Theorem 7.
\end{proof}

\begin{corollary}
Let $\mu $\ be a fuzzy set of $S.$ Then,

(i) $\mu $ is an $\left( \in ,\in \vee q\right) $-fuzzy left(right) ideal of 
$S$ if and only if $\ \mu _{r}(\neq \phi )$ left(right) ideal of $S$ for all 
$r$ $\in (0,0.5].$

(ii) $\mu $ is an $\left( \in ,\in \vee q\right) $-fuzzy left(right) ideal
of $S$ if and only if $\ Q(\mu ;r)(\neq \Phi )$ left(right) ideal of $S$ for
all $r$ $\in (0.5,1].$

(iii) $\mu $ is an $\left( \in ,\in \vee q\right) $-fuzzy left(right) ideal
of $S$ if and only if $\ [\mu ]_{r}(\neq \Phi )$ left(right) ideal of $S$for
all $r$ $\in (0,1].$
\end{corollary}

\begin{theorem}
The intersection of any faimly of $\left( \in _{\gamma },\in _{\gamma }\vee
q_{\delta }\right) $-fuzzy left(right) ideals is an $\left( \in _{\gamma
},\in _{\gamma }\vee q_{\delta }\right) $-fuzzy left(right) ideal.
\end{theorem}

\begin{proof}
Let $\{\mu _{i}\}_{i\in I}$ \ be a faimly of $\left( \in _{\gamma },\in
_{\gamma }\vee q_{\delta }\right) $-fuzzy left ideals of $S$ and $a,s\in S.$
Then, 
\begin{eqnarray*}
((\tbigwedge\limits_{i\in I}\mu _{i})(sa))\vee \gamma
&=&(\tbigwedge\limits_{i\in I}\mu _{i}(sa))\vee \gamma \\
&=&(\tbigwedge\limits_{i\in I}((\mu _{i}(sa))\vee \gamma )) \\
&\geq &(\tbigwedge\limits_{i\in I}\{\mu _{i}(a)\wedge \delta \}) \\
&=&(\tbigwedge\limits_{i\in I}\mu _{i}(a))\wedge \delta \\
&=&(\tbigwedge\limits_{i\in I}\mu _{i})(x)\wedge \delta .
\end{eqnarray*}%
Hence, $\tbigwedge\limits_{i\in I}\mu _{i}$ is an $\left( \in _{\gamma },\in
_{\gamma }\vee q_{\delta }\right) $-fuzzy left ideal of $S.$ (Similarly, we
can prove for right ideals).
\end{proof}

\begin{lemma}
The union of any faimly of $\left( \in _{\gamma },\in _{\gamma }\vee
q_{\delta }\right) $-fuzzy left(right) ideals is an $\left( \in _{\gamma
},\in _{\gamma }\vee q_{\delta }\right) $-fuzzy left(right) ideals .
\end{lemma}

\begin{proof}
Let $\{\mu _{i}\}_{i\in I}$ \ be a faimly of $\left( \in _{\gamma },\in
_{\gamma }\vee q_{\delta }\right) $-fuzzy left ideals of $S$ and $s,a\in S.$
Then, $\left( \bigvee\limits_{i\in I}\mu _{i}\right)
(sa)=\bigvee\limits_{i\in I}(\mu _{i}(ab)).$ Since each $\mu _{i}$ is fuzzy
left ideals of $S,$ so $\mu _{i}(sa)\vee \gamma \geq \mu _{i}(a)\wedge
\delta ,$ for all $i\in I.$ Thus,%
\begin{eqnarray*}
\left( \bigvee\limits_{i\in I}\mu _{i}\right) (sa)\vee \gamma
&=&\bigvee\limits_{i\in I}\left( \mu _{i}(sa)\vee \gamma \right) \\
&\geq &\bigvee\limits_{i\in I}\left( \mu _{i}(a)\wedge \delta \right) \\
&=&\left( \bigvee\limits_{i\in I}\mu _{i}(a)\right) \wedge \delta \\
&=&\left( \bigvee\limits_{i\in I}\mu _{i}\right) (a)\wedge \delta .
\end{eqnarray*}

Hence, $\bigvee\limits_{i\in I}\mu _{i}$ is an $\left( \in _{\gamma },\in
_{\gamma }\vee q_{\delta }\right) $-fuzzy left ideal of $S.$ (Similarly, we
can prve for right ideals).
\end{proof}

\begin{definition}
A fuzzy subset $\mu $ of an LA-semigroup $S$ is called \ an \ $\left( \in
_{\gamma },\in _{\gamma }\vee q_{\delta }\right) $-fuzzy generalized bi-
ideal if for all $a,b,s\in S$ and $t,r\in (\gamma ,1],$ $a_{t}\in _{\gamma
}\mu ,b_{r}\in _{r}\mu $ implies that $\left( \left( as\right) b\right)
_{t\wedge r}\in _{\gamma }\vee q_{\delta }\mu .$
\end{definition}

\begin{remark}
Every fuzzy generalized bi-ideal and every $\left( \in ,\in \vee q\right) $%
-fuzzy generalized bi- ideal is an $\left( \in _{\gamma },\in _{\gamma }\vee
q_{\delta }\right) $-fuzzy generalized bi-ideal but the converse is not true.
\end{remark}

\begin{example}
\begin{equation*}
\begin{tabular}{l|llll}
${\ast }$ & $1$ & $2$ & $3$ & $4$ \\ \hline
$1$ & $3$ & $2$ & $3$ & $2$ \\ 
$2$ & $2$ & $2$ & $2$ & $2$ \\ 
$3$ & $2$ & $2$ & $2$ & $2$ \\ 
$4$ & $2$ & $2$ & $2$ & $2$%
\end{tabular}%
\end{equation*}%
`Define a fuzzy subset $\mu $ of $S$ by

$\mu (1)=0.4,$ $\mu (2)=0.5,$ $\mu (3)=0.35$ and $\mu (4)=0.$

Then,

$(i)$\ $\mu $ is an $\left( \in _{0.3},\in _{0.3}\vee q_{0.35}\right) $%
-fuzzy generalized bi-ideal.

$\left( ii\right) $\ $\mu $ is not an $\left( \in ,\in \vee q\right) $-fuzzy
generalized bi-ideal\ because $1_{0.4}\in \mu ,$ but $\left( \left( 1\ast
3\right) \ast 1\right) _{0.6}\overline{\in \vee q}\mu .$

$\left( iii\right) $\ $\mu $ is not fuzzy generalized bi-ideal.
\end{example}

\begin{theorem}
Let $A$ be a non-empty subset of $S.$ Then, $A$ is a generalized bi-ideal of 
$S$ if and only if the fuzzy subset $\mu $\ of $S$ defined by%
\begin{equation*}
\mu (a)=\QDATOPD\{ . {\geq \delta \text{ if }a\in A}{\leq \gamma \text{ if }%
a\notin A}
\end{equation*}%
\ is an $\left( \in _{\gamma },\in _{\gamma }\vee q_{\delta }\right) $-fuzzy
generalized bi-ideal of $S.$
\end{theorem}

\begin{proof}
Proof is similar to the proof of Theorem 4.
\end{proof}

\begin{corollary}
Let $A$ be a non-empty subset of an LA-semigroup $S.$ Then, $A$ is a
generalized bi-ideal of $S$ if and only if $\chi _{A},$ the characteristic
function of $A$ is an $\left( \in _{\gamma },\in _{\gamma }\vee q_{\delta
}\right) $-fuzzy generalized bi-ideal of $S.$
\end{corollary}

\begin{theorem}
Let $2\delta =1+\gamma $ and $A$ be an generalized bi-ideal of $S.$ Then,
the fuzzy subset $\mu $\ of $S$ defined by%
\begin{equation*}
\mu (a)=\QDATOPD\{ . {\geq \delta \text{ if }a\in A}{\leq \gamma \text{ if }%
a\notin A}
\end{equation*}%
\ is an $\left( q_{\delta },\in _{\gamma }\vee q_{\delta }\right) $-fuzzy
generalized bi-ideal of $S.$
\end{theorem}

\begin{proof}
Proof is similar to the proof of Theorem 5$.$
\end{proof}

\begin{theorem}
A fuzzy subset $\mu $ of an LA-semigroup $S$ is an $\left( \in _{\gamma
},\in _{\gamma }\vee q_{\delta }\right) $-fuzzy generalised bi-ideal of $S$
if \ and only if for all $a,b,s\in S$ and $t,r\in (\gamma ,1],$\ $\mu \
(\left( as\right) b)\vee \gamma \geq \mu \left( a\right) \wedge \mu \left(
b\right) \wedge \delta .$
\end{theorem}

\begin{proof}
Proof is similar to the proof of Theorem 6.
\end{proof}

\begin{remark}
For any \ $\left( \in _{\gamma },\in _{\gamma }\vee q_{\delta }\right) $-
fuzzy generalized bi-ideal $\mu $ of $S$ we can conclude

(i)if $\gamma =0,$ $\delta =1$, then $\mu $ is fuzzy generalized bi-idealof $%
S.$

(ii)if $\gamma =0,$ $\delta =0.5,$ then $\mu $ is an $\left( \in ,\in \vee
q\right) $-fuzzy generalized bi-ideal of $S.$
\end{remark}

\begin{theorem}
Let $\mu $ be a fuzzy subset of $S$. Then,

(i) $\mu $ is an $\left( \in _{\gamma },\in _{\gamma }\vee q_{\delta
}\right) $-fuzzy generalised bi-ideal of $S$ if and only if $\ \mu
_{r}^{\gamma }(\neq \phi )$ is generalised bi-ideal of $S$ for all $r$ $\in
(\gamma ,\delta ].$

(ii) If $2\delta =1+\gamma ,$ then $\mu $ is an $\left( \in _{\gamma },\in
_{\gamma }\vee q_{\delta }\right) $-fuzzy generalised bi-ideal of of $S$ if
and only if $\ \mu _{r}^{\delta }(\neq \phi )$ is generalised bi-ideal of of 
$S$ for all $r$ $\in (\delta ,1].$

(iii) If $2\delta =1+\gamma ,$ then $\mu $ is an $\left( \in _{\gamma },\in
_{\gamma }\vee q_{\delta }\right) $-fuzzy generalised bi-ideal of of $S$ if
and only if $\ [\mu ]_{r}^{\delta }(\neq \phi )$ is generalised bi-ideal of $%
S$ for all $r$ $\in (\gamma ,1].$

(iv) $\mu $ is an $\left( \in _{\gamma },\in _{\gamma }\vee q_{\delta
}\right) $-fuzzy generalized bi-ideal of $S$ if and only if $\ U(\mu
;r)(\neq \phi )$ is a generalized bi-ideal of $S$ for all $r$ $\in (\gamma
,\delta ].$
\end{theorem}

\begin{proof}
Proof is similar to the proof of Theorem 7.
\end{proof}

If we take $\gamma =0$ and $\delta =0.5$ in above \ theorem, we can conclude
the following results

\begin{corollary}
Let $\mu $\ be a fuzzy set of $S.$ Then,

(i) $\mu $ is an $\left( \in ,\in \vee q\right) $-fuzzy generalised bi-ideal
of $S$ if and only if $\ \mu _{r}(\neq \phi )$ is a generalised bi-ideal of $%
S$ for all $r$ $\in (0,0.5].$

(ii) $\mu $ is an $\left( \in ,\in \vee q\right) $-fuzzy generalised
bi-ideal of $S$ if and only if $\ Q(\mu ;r)(\neq \Phi )$ is a generalised
bi-ideal of $S$ for all $r$ $\in (0.5,1].$

(iii) $\mu $ is an $\left( \in ,\in \vee q\right) $-fuzzy generalised
bi-ideal of $S$ if and only if $\ [\mu ]_{r}(\neq \Phi )$ is a generalised
bi-ideal of $S$ for all $r$ $\in (0,1].$

(iv) $\mu $ is an $\left( \in ,\in \vee q\right) $-fuzzy generalized
bi-ideal of $S$ if and only if $U(\mu ;r)(\neq \phi )$ is a generalized
bi-ideal of $S$ for all $r$ $\in (0,0.5].$
\end{corollary}

\begin{definition}
A fuzzy subset $\mu $ of an LA-semigroup $S$ is called \ an \ $\left( \in
_{\gamma },\in _{\gamma }\vee q_{\delta }\right) $ fuzzy bi- ideal if for
all $a,b,s\in S$ and $t,r\in (\gamma ,1]$ $\ a_{t}\in _{\gamma }\mu ,$ $%
b_{r}\in _{r}\mu $ implies that

(i) $\left( ab\right) _{t\wedge r}\in _{\gamma }\vee q_{\delta }\mu .$

(ii) $\left( \left( as\right) b\right) _{t\wedge r}\in _{\gamma }\vee
q_{\delta }\mu .$
\end{definition}

\begin{remark}
Every fuzzy bi- ideal and every $\left( \in ,\in \vee q\right) $-fuzzy bi-
ideal is an $\left( \in _{\gamma },\in _{\gamma }\vee q_{\delta }\right) $%
-fuzzy bi- ideal but the converse is not true.
\end{remark}

\begin{example}
\begin{equation*}
\begin{tabular}{l|llll}
${\ast }$ & $1$ & $2$ & $3$ & $4$ \\ \hline
$1$ & $3$ & $2$ & $3$ & $2$ \\ 
$2$ & $2$ & $2$ & $2$ & $2$ \\ 
$3$ & $2$ & $2$ & $2$ & $2$ \\ 
$4$ & $2$ & $2$ & $2$ & $2$%
\end{tabular}%
\end{equation*}%
`Define fuzzy subset $\mu $ of $S$ by

$\mu (1)=0.4,\mu (2)=0.5,\mu (3)=0.35,\mu (4)=0.$

Then,

$(i)$ $\mu $ is an $\left( \in _{0.3},\in _{0.3}\vee q_{0.35}\right) $-fuzzy
bi-ideal

$\left( ii\right) $ $\mu $ is not an $\left( \in ,\in \vee q\right) $-fuzzy
bi-ideal\ because $1_{0.4}\in \mu ,$ but $\left( \left( 1\ast 3\right) \ast
1\right) _{0.6}\overline{\in \vee q}\mu $

$\left( iii\right) $ $\mu $ is not fuzzy bi-ideal
\end{example}

\begin{theorem}
Let $A$ be a non-empty subset of $S.$ Then, $A$ is a bi-ideal of $S$ if and
only if the fuzzy subset $\mu $\ of $S$ defined by%
\begin{equation*}
\mu (a)=\QDATOPD\{ . {\geq \delta \text{ if }a\in A}{\leq \gamma \text{ if }%
a\notin A}
\end{equation*}%
\ is an $\left( \in _{\gamma },\in _{\gamma }\vee q_{\delta }\right) $-fuzzy
bi-ideal of $S.$
\end{theorem}

\begin{proof}
Proof is similar to the proof of Theorem 4.
\end{proof}

\begin{corollary}
Let $A$ be a non-empty subset of an LA-semigroup $S.$ Then, $A$ is a
bi-ideal of $S$ if and only if $\chi _{A},$ the characteristic function of $%
A $ is an $\left( \in _{\gamma },\in _{\gamma }\vee q_{\delta }\right) $%
-fuzzy bi-ideal of $S.$
\end{corollary}

\begin{theorem}
Let $2\delta =1+\gamma $ and $A$ be a bi-ideal of $S.$ Then, the fuzzy
subset $\mu $\ of $S$ defined by%
\begin{equation*}
\mu (a)=\QDATOPD\{ . {\geq \delta \text{ if }a\in A}{\leq \gamma \text{ if }%
a\notin A}
\end{equation*}%
\ is an $\left( q_{\delta },\in _{\gamma }\vee q_{\delta }\right) $-fuzzy
bi-ideal of $S.$
\end{theorem}

\begin{proof}
Proof is similar to the proof of Theorem 5.
\end{proof}

\begin{theorem}
A fuzzy subset $\mu $ of an LA-semigroup $S$ is an $\left( \in _{\gamma
},\in _{\gamma }\vee q_{\delta }\right) $-fuzzy bi-ideal of $S$ if and only
if for all $a,b,s\in S$ and $t,r\in (\gamma ,1]$

(i) $\mu \ (ab)\vee \gamma \geq \mu \left( a\right) \wedge \mu \left(
b\right) \wedge \delta $

(ii)\ $\mu \ (\left( as\right) b)\vee \gamma \geq \mu \left( a\right) \wedge
\mu \left( b\right) \wedge \delta .$
\end{theorem}

\begin{proof}
Proof is similar to the proof of Theorem 6.
\end{proof}

\begin{remark}
For any \ $\left( \in _{\gamma },\in _{\gamma }\vee q_{\delta }\right) $-
fuzzy bi-ideal $\mu $ of $S$ we can conclude

(i) if $\gamma =0,\delta =1$, then $\mu $ is fuzzy bi-idealof $S.$

(ii) if $\gamma =0,\delta =0.5,$ then $\mu $ is an $\left( \in ,\in \vee
q\right) $-fuzzy bi-ideal of $S.$
\end{remark}

\begin{theorem}
Let $\mu $ be a fuzzy subset of $S$. Then,

(i) $\mu $ is an $\left( \in _{\gamma },\in _{\gamma }\vee q_{\delta
}\right) $-fuzzy bi-ideal of $S$ if and only if $\ \mu _{r}^{\gamma }(\neq
\phi )$ is bi-ideal of $S$ for all $r$ $\in (\gamma ,\delta ].$

(ii) If $2\delta =1+\gamma .$ Then $\mu $ is an $\left( \in _{\gamma },\in
_{\gamma }\vee q_{\delta }\right) $-fuzzy bi-ideal of of $S$ if and only if $%
\ \mu _{r}^{\delta }(\neq \phi )$ is bi-ideal of of $S$ for all $r$ $\in
(\delta ,1].$

(iii) If $2\delta =1+\gamma .$ Then $\mu $ is an $\left( \in _{\gamma },\in
_{\gamma }\vee q_{\delta }\right) $-fuzzy bi-ideal of of $S$ if and only if $%
\ [\mu ]_{r}^{\delta }(\neq \phi )$ is bi-ideal of $S$ for all $r$ $\in
(\gamma ,1].$

(iv) $\mu $ is an $\left( \in _{\gamma },\in _{\gamma }\vee q_{\delta
}\right) $-fuzzy bi-ideal of $S$ if and only if $\ U(\mu ;r)(\neq \phi )$ is
a bi-ideal of $S$ for all $r$ $\in (\gamma ,\delta ].$
\end{theorem}

\begin{proof}
Proof is similar to the proof of Theorem 7.
\end{proof}

If we take $\gamma =0$ and $\delta =0.5$ in above \ theorem, we can conclude
the following results

\begin{corollary}
Let $\mu $ be a fuzzy set of $S.$Then

(i) $\mu $ is an $\left( \in ,\in \vee q\right) $-fuzzy bi-ideal of $S$ if
and only if $\ \mu _{r}(\neq \phi )$ is bi-ideal of $S$ for all $r$ $\in
(0,0.5].$

(ii) $\mu $ is an $\left( \in ,\in \vee q\right) $-fuzzy bi-ideal of $S$ if
and only if $\ Q(\mu ;r)(\neq \Phi )$ is bi-ideal of $S$ for all $r$ $\in
(0.5,1].$

(iii) $\mu $ is an $\left( \in ,\in \vee q\right) $-fuzzy bi-ideal of $S$ if
and only if $\ [\mu ]_{r}(\neq \Phi )$ is bi-ideal of $S$ for all $r$ $\in
(0,1].$

(iv) $\mu $ is an $\left( \in ,\in \vee q\right) $-fuzzy bi-ideal of $S$ if
and only if $U(\mu ;r)(\neq \phi )$ is a bi-ideal of $S$ for all $r$ $\in
(0,0.5].$
\end{corollary}

\begin{definition}
A fuzzy subset $\mu $ of an LA-semigroup $c$ is called an $\left( \in
_{\gamma },\in _{\gamma }\vee q_{\delta }\right) $-fuzzy interior ideal of $%
c $ if for all $a,b,c\in S$ and $t,r\in (\gamma ,1],$ the following
conditions hold:

$(i)$ $a_{t},b_{r}\in _{\gamma }\mu $ implies that $\left( ab\right)
_{t\wedge r}\in _{\gamma }\vee q_{\delta }\mu .$

$(ii)$ $c_{t}\in _{\gamma }\mu $ implies that $\left( \left( ac\right)
b\right) _{t}\in _{\gamma }\vee q_{\delta }\mu .$
\end{definition}

\begin{theorem}
Let $A$ be a non empty subset of $S.$ Then, $A$ is a interior ideal of $S$
if and only if the fuzzy subset $\mu $\ of $S$ defined by%
\begin{equation*}
\mu (x)=\QDATOPD\{ . {\geq \delta \text{ if }x\in A}{\leq \gamma \text{ if }%
x\notin A}
\end{equation*}%
\ is an $\left( \in _{\gamma },\in _{\gamma }\vee q_{\delta }\right) $-fuzzy
interior ideal of $S.$
\end{theorem}

\begin{proof}
Proof is similar to the proof of Theorem 4.
\end{proof}

\begin{corollary}
Let $A$ be a non-empty subset of $S.$ Then, $A$ is a interior ideal of $S$
if and only if $\chi _{A},$ the characteristic function of $A$ is an $\left(
\in _{\gamma },\in _{\gamma }\vee q_{\delta }\right) $-fuzzy interior ideal
of $S.$
\end{corollary}

\begin{theorem}
Let $2\delta =1+\gamma $ and $A$ be a bi-ideal of $S.$ Then, the fuzzy
subset $\mu $\ of $S$ defined by%
\begin{equation*}
\mu (a)=\QDATOPD\{ . {\geq \delta \text{ if }a\in A}{\leq \gamma \text{ if }%
a\notin A}
\end{equation*}%
\ is an $\left( q_{\delta },\in _{\gamma }\vee q_{\delta }\right) $-fuzzy
bi-ideal of $S.$
\end{theorem}

\begin{proof}
Proof is similar to the proof of Theorem 5.
\end{proof}

\begin{theorem}
A fuzzy subset $\mu $ of an LA-semigroup $S$ is an $\left( \in _{\gamma
},\in _{\gamma }\vee q_{\delta }\right) $-fuzzy interior ideal of $S$ if and
only if \ for all $a,b,c\in S,$

(i)$\ \mu (ab)\vee \gamma \geq \mu (a)\wedge \mu (b)\wedge \delta .$

(ii) $\mu (\left( ac\right) b)\vee \gamma \geq \mu (c)\wedge \delta .$
\end{theorem}

\begin{proof}
Proof is similar to the proof of Theorem 6.
\end{proof}

\begin{theorem}
Let $\mu $ be a fuzzy subset of $S$. Then,

(i) $\mu $ is an $\left( \in _{\gamma },\in _{\gamma }\vee q_{\delta
}\right) $-fuzzy interior ideal of $S$ if and only if$\ \mu _{r}^{\gamma
}(\neq \phi )$ is a interior ideal of $S$ for all $r$ $\in (\gamma ,\delta
]. $

(ii) If $2\delta =1+\gamma ,$ then $\mu $ is an $\left( \in _{\gamma },\in
_{\gamma }\vee q_{\delta }\right) $-fuzzy interior ideal of $S$ if and only
if $\ \mu _{r}^{\delta }(\neq \phi )$ is interior ideal of $S$ for all $r$ $%
\in (\delta ,1].$

(iii) If $2\delta =1+\gamma ,$ then $\mu $ is an $\left( \in _{\gamma },\in
_{\gamma }\vee q_{\delta }\right) $-fuzzy interior ideal of $S$ if and only
if $\ [\mu ]_{r}^{\delta }(\neq \phi )$ is interior ideal of $S$ for all $r$ 
$\in (\gamma ,1].$

(iv) (iv) $\mu $ is an $\left( \in _{\gamma },\in _{\gamma }\vee q_{\delta
}\right) $-fuzzy interior ideal of $S$ if and only if $\ U(\mu ;r)(\neq \phi
)$ is a interior ideal of $S$ for all $r$ $\in (r,\delta ].$
\end{theorem}

\begin{proof}
Proof is similar to the proof of theorem 7.
\end{proof}

If we take $\gamma =0$ and $\delta =0.5$ in above \ theorem, we can conclude
the following results

\begin{corollary}
Let $\mu $\ be a fuzzy set of $c.$ Then,

(i) $\mu $ is an $\left( \in ,\in \vee q\right) $-fuzzy interior ideal of $%
S, $ if and only if $\ \mu _{r}(\neq \phi )$ interior ideal of $S,$ for all $%
r$ $\in (0,0.5].$

(ii) $\mu $ is an $\left( \in ,\in \vee q\right) $-fuzzy interior ideal of $%
S $ if and only if $\ Q(\mu ;r)(\neq \Phi )$ interior ideal of $S,$ for all $%
r$ $\in (0.5,1].$

(iii) $\mu $ is an $\left( \in ,\in \vee q\right) $-fuzzy interior ideal of $%
S$ if and only if $\ [\mu ]_{r}(\neq \Phi )$ interior ideal of $S$,for all $%
r $ $\in (0,1].$

(iv) $\mu $ is an $\left( \in ,\in \vee q\right) $-fuzzy interior ideal of $%
S $ if and only if $\ U(\mu ;r)(\neq \phi )$ is a interior ideal of $S$ for
all $r$ $\in (0,0.5].$
\end{corollary}

\begin{theorem}
Every $\left( \in _{\gamma },\in _{\gamma }\vee q_{\delta }\right) $-fuzzy
ideal of an LA-semigroup is $S$ is an $\left( \in _{\gamma },\in _{\gamma
}\vee q_{\delta }\right) $-fuzzy interior ideal of $S.$
\end{theorem}

\begin{proof}
Let $\mu $ is an $\left( \in _{\gamma },\in _{\gamma }\vee q_{\delta
}\right) $-fuzzy ideal of $S.$ Then, 
\begin{equation*}
\ \mu (xy)\vee \gamma \geq \mu (x)\wedge \delta \geq \mu (x)\wedge \mu
(y)\wedge \delta .
\end{equation*}%
Thus, $\mu $ is an $\left( \in _{\gamma },\in _{\gamma }\vee q_{\delta
}\right) $-fuzzy LA-subsemigroup of $S.$ Also, for all $x,a,y\in S,$ we have$%
\ $%
\begin{eqnarray*}
\mu (\left( xa\right) y)\vee \gamma &=&\left( \mu (\left( xa\right) y)\vee
\gamma \right) \vee \gamma \geq \left( \mu (xa)\wedge \delta \right) \vee
\gamma \\
&=&\left( \mu (xa)\vee \gamma \right) \wedge \delta \geq \mu (a)\wedge
\delta \wedge \delta =\mu (a)\wedge \delta .
\end{eqnarray*}%
Hence, $\mu $\ is an $\left( \in _{\gamma },\in _{\gamma }\vee q_{\delta
}\right) $-fuzzy interior ideal of $S.$
\end{proof}

\begin{definition}
A fuzzy subset $\mu $ of an LA-semigroup $S$ is called an $\left( \in
_{\gamma },\in _{\gamma }\vee q_{\delta }\right) $-fuzzy quasi-ideal of $S,$%
\ if it satisfies,

$\mu (x)\vee \gamma \geq \left( \mu \circ 1\right) (x)\wedge \left( 1\circ
\mu \right) (x)\wedge \delta ,$

where $1$ denotes the fuzzy subset of $S$ mapping every element of $S$ on $1$%
.
\end{definition}

\begin{theorem}
Let $\mu $ be an $\left( \in _{\gamma },\in _{\gamma }\vee q_{\delta
}\right) $-fuzzy quasi-ideal of $S,$\ then the set $\mu _{\gamma }=\{a\in
S\mid \mu (a)>\gamma \}$ is a quasi-ideal of $S.$
\end{theorem}

\begin{proof}
To prove the required result, we need to show that $S\mu _{\gamma }\cap \mu
_{\gamma }S\subseteq $ $\mu _{\gamma }.$ Let $x\in S\mu _{\gamma }\cap \mu
_{\gamma }S.$ Then, $x\in S\mu _{\gamma }$ and $x\in \mu _{\gamma }S.$ So, $%
x=sa$ and $x=bt$ for some $a,b\in \mu _{\gamma }$ and $s,t\in S.$ Thus, $\mu
(a)>\gamma $ and $\mu (b)>\gamma .$ Since%
\begin{eqnarray*}
\left( 1\circ \mu \right) (x) &=&\bigvee\limits_{x=yz}\{1(y)\wedge \mu (z)\}
\\
&\geq &\{1(s)\wedge \mu (a)\},\text{ because }x=sa \\
&=&\mu (a).
\end{eqnarray*}%
Similarly, $\left( \mu \circ 1\right) (x)\geq \mu (b).$ Thus,%
\begin{eqnarray*}
\mu (x)\vee \gamma &\geq &\left( \mu \circ 1\right) (x)\wedge (1\circ \mu
)(x)\wedge \delta \\
&\geq &\mu (a)\wedge \mu (b)\wedge \delta \\
&>&\gamma \text{ because }\mu (a)>\gamma ,\text{ }\mu (b)>\gamma .
\end{eqnarray*}%
Which implies that $\mu (x)>\gamma .$ Thus, $x\in \mu _{\gamma }.$ Hence, $%
\mu _{\gamma }$ is a quasi-ideal of $S.$
\end{proof}

\begin{remark}
Every fuzzy quasi-ideal and $\left( \in ,\in \vee q\right) $-fuzzy
quasi-ideal of $S$ is an $\left( \in _{\gamma },\in _{\gamma }\vee q_{\delta
}\right) $-fuzzy quasi-ideal of $S$ but the converse is not true.
\end{remark}

\begin{example}
If we consider the LA-semigroup given in example 4. Then, the fuzzy subset $%
\mu $ defined by $\mu (1)=0,$ $\mu (2)=0.2,$ $\mu (3)=0.3,$ $\mu (4)=0$ is

(i) an $\left( \in _{0.1},\in _{0.1}\vee q_{0.2}\right) $-fuzzy quasi-ideal.

(ii) not an $\left( \in ,\in \vee q\right) $-fuzzy quasi-ideal$.$ Indeed $%
\left( \mu \circ 1\right) (2)=0.3=(1\circ \mu )(2)$ but $\mu (2)=0.2\ngeq
\left( \mu \circ 1\right) (2)\wedge (1\circ \mu )(2)\wedge 0.5$.

(iii) not fuzzy quasi-ideal. Because $\mu (2)=0.2\ngeq \left( \mu \circ
1\right) (2)\wedge (1\circ \mu )(2)=0.3.$
\end{example}

\begin{lemma}
A non-empty subset $Q$ of an LA-semigroup $S$ is a quasi-ideal of $S$ if and
only if the characteristic function $\chi _{Q}$ of $Q$ is an $\left( \in
_{\gamma },\in _{\gamma }\vee q_{\delta }\right) $-fuzzy quasi-ideal of $S.$
\end{lemma}

\begin{proof}
Suppose that $Q$ is a quasi-ideal of $S$ and $\chi _{Q}$ is the
characteristic function of $Q.$ If $x\notin Q,$ then $x\notin SQ$ or $%
x\notin QS.$ Thus, $\left( 1\circ \chi _{Q}\right) (x)=0$ or $\left( \chi
_{Q}\circ 1\right) (x)=0$. So, $\left( 1\circ \chi _{Q}\right) (x)\wedge
\left( \chi _{Q}\circ 1\right) (x)\wedge \delta =0\leq \chi _{Q}(x)\vee
\gamma .$ If $x\in Q,$ then $\chi _{Q}(x)\vee \gamma =1\vee \gamma =1\geq
\left( 1\circ \chi _{Q}\right) (x)\wedge \left( \chi _{Q}\circ 1\right)
(x)\wedge \delta .$ Thus, $\chi _{Q}$ is an $\left( \in _{\gamma },\in
_{\gamma }\vee q_{\delta }\right) $-fuzzy quasi-ideal of $S.$ Conversely,
assume that $\chi _{Q}$ is an $\left( \in _{\gamma },\in _{\gamma }\vee
q_{\delta }\right) $-fuzzy quasi-ideal of $S.$ Let $x\in QS\cap SQ,$ then
there exists $s,t\in S$ and $a,b\in Q,$ such that, $x=as$ and $x=tb.$ Now,%
\begin{eqnarray*}
\left( \chi _{Q}\circ 1\right) (x) &=&\bigvee\limits_{x=yz}\{\chi
_{Q}(y)\wedge 1(z)\} \\
&\geq &\{\chi _{Q}(a)\wedge 1(s)\}\text{ because }x=as \\
&=&1\wedge 1 \\
&=&1.
\end{eqnarray*}%
So, $\left( \chi _{Q}\circ 1\right) (x)=1.$ Similarly, $\left( 1\circ \chi
_{Q}\right) (x)=1.$ Hence, $\left( \chi _{Q}\right) (x)\vee \gamma \geq
\left( 1\circ \chi _{Q}\right) (x)\wedge \left( \chi _{Q}\circ 1\right)
(x)\wedge \delta =1\wedge 1\wedge \delta =\delta .$ Thus, $\left( \chi
_{Q}\right) (x)\vee \gamma \geq \delta .$ This implies that $\left( \chi
_{Q}\right) (x)\geq \delta $ because $\gamma <\delta .$ So, $\left( \chi
_{Q}\right) (x)=1.$ Hence, $x\in Q.$ Thus, $QS\cap SQ\subseteq Q.$ Hence, $Q$
is quasi-ideal of $S.$
\end{proof}

\begin{theorem}
Every $\left( \in _{\gamma },\in _{\gamma }\vee q_{\delta }\right) $-fuzzy
left(right) ideal $\mu $ of $S$ is an $\left( \in _{\gamma },\in _{\gamma
}\vee q_{\delta }\right) $-fuzzy quasi-ideal of $S.$
\end{theorem}

\begin{proof}
Let $a\in S,$ then%
\begin{equation*}
\left( 1\circ \mu \right) (a)=\bigvee\limits_{a=xy}\{1(x)\wedge \mu
(y)\}=\bigvee\limits_{a=xy}\mu (y).
\end{equation*}%
This implies that%
\begin{eqnarray*}
\left( 1\circ \mu \right) (a)\wedge \delta &=&\left(
\bigvee\limits_{a=xy}\mu (y)\right) \wedge \delta \\
&=&\bigvee\limits_{a=xy}\{\mu (y)\wedge \delta \} \\
&\leq &\bigvee\limits_{a=xy}\{\mu (xy)\vee \gamma \}\text{ (because }\mu 
\text{ is an }\left( \in _{\gamma },\in _{\gamma }\vee q_{\delta }\right) 
\text{-fuzzy left ideal of }S.\text{)} \\
&=&\mu (a)\vee \gamma .
\end{eqnarray*}%
Thus, $\left( 1\circ \mu \right) (a)\wedge \delta \leq \mu (a)\vee \gamma .$
Hence, $\mu (a)\vee \gamma \geq \left( 1\circ \mu \right) (a)\wedge \delta
\geq \left( \mu \circ 1\right) (a)\wedge (1\circ \mu )(a)\wedge \delta .$
Thus, $\mu $ is an $\left( \in _{\gamma },\in _{\gamma }\vee q_{\delta
}\right) $-fuzzy quasi-ideal of $S.$ (Similarly, we can prove for right
ideal).
\end{proof}

\end{document}